\documentclass{daj}
\usepackage[utf8]{inputenc}
\usepackage{amsmath,amsfonts,amssymb,amsthm,mathrsfs,dsfont}
\usepackage{minitoc}

\renewcommand{\P}{\mathbb{P}}
\newcommand{\E}{\mathbb{E}}
\newcommand{\R}{\mathbb{R}}
\DeclareFixedFont{\beaupetit}{T1}{ftp}{b}{n}{2cm}

\newcommand{\Ind}[1]{\mathds{1}_{#1}}%Indicatrice

\newtheorem{theorem}{Theorem}[]

\newtheorem{proposition}[]{Proposition}

\newtheorem{lemma}[]{Lemma}

\theoremstyle{definition}

\newtheorem*{remark}{Remark}

\dajAUTHORdetails{%
  title = {The Phase Transition for Parking on Galton--Watson trees}, %% please capitalize all significant words
  author = {Nicolas Curien and Olivier H\'enard},
  plaintextauthor = {Nicolas Curien and Olivier H\'enard},
  plaintexttitle = {The phase transition for parking on Galton--Watson trees}, %%  title without math or LaTeX
  runningtitle = {The phase transition for parking on Galton--Watson trees}, 
  runningauthor = {Nicolas Curien, Olivier H\'enard},
  copyrightauthor = {N. Curien, O. H\'enard}
}   %%% END \dajAUTHORdetails

\dajEDITORdetails{%
   year={2022},
   number={1},
   received={16 December 2019},   % received date: example: 7 January 2017
   revised={2 February 2022},    % Optional revised date (you may comment it out)
   published={10 March 2022},  % published date
   doi={10.19086/da.33167},       % XXX = number of paper, e.g. da006 for paper#6
}   %%% END \dajEDITORdetails

\begin{document}

\begin{frontmatter}[classification=text]
%% EDITOR: this will force the keywords to appear right after the Abstract.
%%   If the abstract is too long and would force the keywords off the
%%   front page, please comment out % [classification=text] above
%%   This way the keywords will be floated on the bottom of the first page
%%   even though the Abstract spills over to the next page.

%%% AUTHOR: Title goes here.  This line is optional.  You must use it
%%   if title has footnote attached or requires nontrivial typesetting,
%%   e.g., inclusion of linebreaks to force nice layout.
\title{The Phase Transition for Parking on Galton--Watson Trees} %% please capitalize all significant words

%%% AUTHOR:
%%% List all authors. If you wish, place grant acknowledgements in \thanks.
%%% In brackets include a short tag for each author.
\author[nicolas]{Nicolas Curien \thanks{Supported by ERC Advanced Grant 740943 ``GeoBrown''. }}
\author[olivier]{Olivier H\'enard}

%%% AUTHOR: Abstract goes here
\begin{abstract}
We establish a phase transition for the parking process on critical Galton--Watson trees. In this model, a random number of cars with mean $m$ and variance $\sigma^{2}$ arrive independently on the vertices of a critical Galton--Watson tree with finite variance $\Sigma^{2}$ conditioned to be large. The cars go down the tree towards the root and try to park on empty vertices as soon as possible. We show a phase transition depending on 
$$ \Theta:= (1-m)^2- \Sigma^2 (\sigma^2+m^2-m).$$ 
Specifically, when $m \leq 1$, if $ \Theta>0,$ then all but (possibly) a few cars will manage to park, whereas if $\Theta<0$,
then a positive fraction of the cars will not find a spot and exit the tree through the root. This confirms a conjecture of Goldschmidt and Przykucki \cite{GP19}.
\end{abstract}
\end{frontmatter}

%%% AUTHOR: body of paper starts here
%\section{Introduction}
 \begin{figure}[!h]
 \begin{center}
 \includegraphics[width=0.8\linewidth]{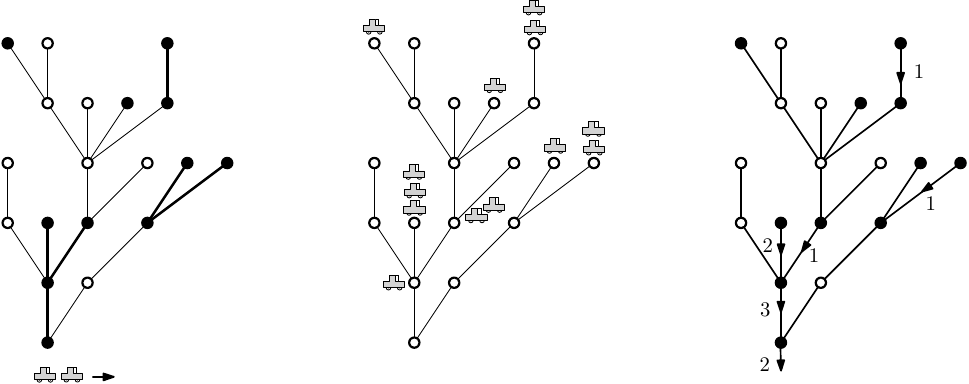}
 \caption{Middle: A plane tree together with a configuration of cars trying to park. Left: the resulting parking configuration, where $2$ cars did not manage to park on the tree. Right: the same parked tree with flux on the edges (all the non-labeled edges have flux zero).}
 \end{center}
 \end{figure}
 
\section{Introduction}
The parking process on the line is a very classical problem in probability and combinatorics.  Recently, a generalization of this process  on plane trees received much attention \cite{LaP16,GP19,CG19,JO18,HMP19}. In this paper, we shall see, in full generality, that this process displays a rich phase transition phenomenon (sharing many similarities with the usual phase transition for Bernoulli percolations on deterministic lattices) and we pinpoint the location of the phase transition which depends only on the means and variances of the car arrivals and on the critical offspring distribution of the underlying Galton--Watson trees, thereby confirming a conjecture of Goldschmidt and Przykucki.

\paragraph{Parking on a rooted plane tree.} We consider a finite plane\footnote{The planar embedding plays no role in the parking procedure; still, it is a convenient setting since it breaks annoying symmetries, enables one to define unambiguously Galton--Watson trees, and allows the use of a great many tools, e.g. spinal decompositions.}rooted tree $\mathfrak{t}$ whose vertices will be interpreted as free parking spots, each spot accommodating at most 1 car, together with a configuration $\ell :  \mathrm{Vertices}( \mathfrak{t}) \to \{0,1,2, \dots \}$ representing the number of cars arriving on each vertex. Each car tries to park on its arrival vertex, and if the spot is occupied, it travels downward towards the root of the tree until it finds an empty vertex to park. If there is no such vertex on its way, the car exits the tree through the root $ \varnothing$. The outgoing \emph{flux} $\varphi(\mathfrak{t}, \ell)$ is the number of cars which did not manage to park. Let us note two important properties of the model. First, the final configuration and the outgoing flux do not depend upon the order chosen to park the cars: we call it the Abelian property of the model.  Second, we have a monotonicity property: the outgoing flux is an increasing function of $\ell$ for a given tree $ \mathfrak{t}$.

Our stochastic model of parking is as follows. Given a (random) rooted plane tree $ \mathfrak{t}$, we shall suppose that the arrivals of cars on each vertex of $ \mathfrak{t}$ are independent identically distributed random variables with law 
\begin{align*}  \mu \quad \mbox{ with mean }m >0  \mbox{ and finite variance } \sigma^{2} &&&\mbox{ \textbf{(car arrivals)}}.  \end{align*} That is, conditionally on $ \mathfrak{t}$, the variables $(\ell(x))_{x\in \mathrm{Vertices}( \mathfrak{t})}$ are i.i.d.~with law $\mu$. By abuse of notation, {in the rest of this paper we shall always deal with trees with i.i.d.~labels and do not specify it further, e.g. we shall write $\varphi( \mathfrak{t}) \equiv \varphi( \mathfrak{t}, \ell)$ for the (random) outgoing flux of cars}. In what follows, the random tree $ \mathfrak{t}$ will be a version of a critical Bienaym\'e--Galton--Watson tree with offspring distribution 
 \begin{align*} \nu \quad \mbox{aperiodic\footnotemark{ }   with mean }1 \mbox{ and finite variance } \Sigma^{2} &&&\mbox{ \textbf{(offspring distribution)}} . \end{align*}

%\addtocounter{footnote}{-1}
\footnotetext{Aperiodic here means that the period of $\nu$ defined by $\max\{k \in \mathbb N ; \nu(k \mathbb N) = 1\}$ is equal to 1.}
%\stepcounter{footnote}

Specifically, we shall consider the parking process on three different types of random trees: the Galton--Watson tree $ \mathcal{T}$, its version $\mathcal{T}_{n}$ conditioned to have $n$ vertices, and the weak local limit $\mathcal{T}_{\infty}$ of the family $(\mathcal{T}_{n})_{n\geq 1}$, which one may also regard as the original Galton--Watson tree $\mathcal T$ conditioned to survive forever. Both distributions $\mu, \nu$ will be taken distinct from  $\delta_{1}$ without further notice.
 
 \paragraph{The phase transition.} Our main result establishes a sharp phase transition for the parking process on these random trees. To describe it, let us first focus on the case of $ \mathcal{T}_{n}$ for large $n$. Heuristically if the ``density" of cars is small enough, we expect that most of them can park on $ \mathcal{T}_{n}$ and the outgoing flux should be small : we can still have local conflicts near the root of the tree so some cars may not manage to park. %, depending only on a neighborhood of the root. 
 On the other hand, if there are ``too many" cars, then we expect that a positive fraction of the cars will not park, hence $\varphi( \mathcal{T}_{n})$ is asymptotically linear in $n$.
This is indeed the case: 

 \begin{theorem}[Phase transition for parking on Galton--Watson trees] \label{thm:main} If $m$ and $\sigma^{2}$ respectively are the mean and variance of $\mu$ (car arrivals),  and if  $\Sigma^{2}$ is the variance of the critical offspring distribution $\nu$, then we let $$\Theta:=(1-m)^2- \Sigma^2 (\sigma^2 +m^2-m).$$
Assuming $m \leq 1$, and $\mu, \nu \neq \delta_1$, we have three regimes classified as follows:
 \begin{center}
 \begin{tabular}{|c||l|c|c|c|}
 \hline
&  & subcritical  
& critical
  & supercritical \\
& & $ \Theta>0$ 
& $ \Theta=0 $ 
& $ \Theta<0 $\\
  \hline\hline
(i) &  $ \varphi( \mathcal{T}_{n})$ as $n \to \infty$ & $\mbox{converges in distribution}$  
& $ \xrightarrow[n\to\infty]{ (\mathbb{P})}\infty, \mbox{ but is }o_{\P}(n)$ 
& $ \approx cn, \quad c>0$\\
\hline \hline
 % (ii) & $\Theta\equiv \Theta( m, \sigma, \Sigma)$ & $>0$ & $0$ & $<0$ \\   \hline       \hline
 (ii) &$ \Sigma^2 \E[\varphi (\mathcal{T})] + m-1$& $- \sqrt{\Theta}$ 
 & $0$ 
 & $\infty$\\
  \hline  
      \hline
(iii)  &$ \mathbb{P}( \mbox{a car is parked at } {\varnothing} \mbox{ in } \mathcal{T})$ & $m$ 
& $m$ 
& $m-c$  \\
\hline
\hline
 \end{tabular}
 \end{center}
where $c = c (\mu,\nu) \in (0,m)$ is a deterministic number.
 \end{theorem}
 \medskip

 Let us comment on our result. The first  line $(i)$ of the above table shows that there is indeed a phase transition for the outgoing flux $\varphi( \mathcal{T}_{n})$ as $n \to \infty$ :  the flux jumps from values of order $O_\P(1)$ to\footnote{the former means that the flux satisfies $\sup_{n} \P(\varphi( \mathcal{T}_{n}) \geq K ) \to 0$  as $K \to \infty$, and the latter that $\varphi(\mathcal{T}_n)/n$ converges to  $c$ in probability} $\approx cn$ in a small variation of the parameter $\Theta=\Theta(m, \sigma^2, \Sigma^2)$. The assumption $m \leq 1$ is not demanding since otherwise the model is clearly supercritical (there are typically more cars than parking spots !).  By $(ii)$ this transition also coincides with the moment where the mean flux at the root of $ \mathcal{T}$ jumps from a finite to an infinite value. The effect is even more dramatic (and easier to analyse) on the infinite tree $ \mathcal{T}_{\infty}$. By the classical spinal decomposition, see \cite[Chapter 12.1]{LP16}, $ \mathcal{T}_{\infty}$ is obtained by grafting independently on each vertex of a semi-infinite line a random number $Y-1$ of (unconditioned) Galton--Watson trees $ \mathcal{T}$   where $Y$ follows the size-biased distribution $\overline{\nu}_k := k \nu_k$ for $k \geq 1$. It follows that the law of $\varphi( \mathcal{T}_{\infty})$ is simply related with the one of the supremum of the random walk with i.i.d.~increments with law 
 $$ \sum_{i=1}^{Y-1} F_{i} + P-1,$$
where $ Y \sim \overline{\nu}$, $F_{i} \sim \varphi( \mathcal{T})$ and $P \sim \mu$ are all independent, see Equation \eqref{eq:supZ} for details. In particular, from line $(ii)$ of the previous table we see that this random walk has a negative drift in the subcritical regime (and so its supremum is finite), has zero mean in the critical regime, and infinite mean in the supercritical one (and so its supremum is infinite). The last line of the table is connected to the law of large numbers on $ \varphi ( \mathcal{T}_{n})$ via the quenched  convergence of the fringe subtree distribution established by Janson \cite{JA16}, see Lemma \ref{prop:janson-fringe}.

\begin{remark} It may appear as a ``little  miracle" that the location of the phase transition only depends on the first two moments and not on more complicated observables of the underlying distributions. This was indeed conjectured in \cite{GP19} using a non-rigorous variance analysis of $\varphi( \mathcal{T})$. For example, if $\sigma^2=\infty$ or simply if $\sigma^2\Sigma^2\geq 1$ then the model is supercritical regardless of the density $m>0$ of cars.
\end{remark}

\paragraph{Previous works.} To the best of our knowledge, the parking process on random trees was first studied by Lackner \& Panholzer \cite{LaP16} in the case of Poisson car arrivals on Cayley trees where they established a phase transition using involved analytic combinatorics techniques, see also \cite{P20}. This phase transition was further explained by Goldschmidt \& Przykucki \cite{GP19} using the infinite tree $ \mathcal{T}_\infty$. In \cite{GP19} the results are transfered from $ \mathcal{T}_\infty$ to $ \mathcal{T}_n$ using increasing couplings. Their arguments were later applied to the case of geometric plane trees (still with Poisson car arrivals) by Chen \& Goldschmidt \cite{CG19}. Motivated by a hydrological modeling problem, Jones \cite{JO18} independently considered the parking process on random trees in the case of binary arrivals on a binary tree. All these models are  encompassed by our general framework. Notice however that Lackner \& Panholzer \cite{LaP16} and Jones \cite{JO18} got some critical exponents in the critical case.
In a recent preprint \cite{CC20}, the uniform parking on a uniform Cayley tree (corresponding to $\mu$ and $\nu$ following a Poisson distribution) has been coupled with a variation of the Erd\"os--R\'enyi random graph. This enables the authors to study the number and sizes of the components in the critical window, but not their geometry.
We believe that the scaling limits of these components should be intimately connected to random growth-fragmentation processes introduced by  Bertoin, which appear in the study of random planar maps \cite{BCK18}. The parking process on trees is also related to the Derrida-Retaux model on supercritical Galton--Watson trees recently tackled in \cite{HMP19} and more generally to recursive $(\min,+)$ distributional equations, see \cite{AB05}.

\medskip 
Contrary to \cite{LaP16,CG19,GP19} which ultimately rely on some explicit computation, our method of proof in this paper is general and purely probabilistic. {It relies on classical tools in percolation theory such as differential (in)equalities obtained through increasing couplings combined with the use of many-to-one lemmas and  spinal decompositions of random trees  (see Eq.~\eqref{eq:spinal}).} %and absolute continuity relation between (parts of) $ \mathcal{T}_{n}$ and $ \mathcal{T}_{\infty}$, see Section \ref{sec:controls}. 
\medskip

\section{The different  trees and their relations}
\label{sec:controls}

We recall here the basic properties of the Galton--Watson tree $ \mathcal{T}$, its version $ \mathcal{T}_n$ conditioned to have $n$ vertices, and of Kesten's tree, the infinite Galton--Watson tree $ \mathcal{T}_\infty$ obtained as the local limit of $ \mathcal{T}_{n}$ as $n \to \infty$. We refer to \cite{AD14} for background on these objects.% Spinal decompositions and the absolute continuity relations between (parts of) these trees are at the core of our arguments.

\subsection{Parking on $ \mathcal{T}_\infty$ and a random walk}

We quickly recall the construction of Kesten's tree $ \mathcal{T}_\infty$. We denote by $\overline{ \nu}$ the size biased distribution of $\nu$ obtained by putting for $k \geq 0$
$$ \overline{\nu}_{k} = k \nu_{k}.$$
By criticality of $\nu$, this defines a probability distribution with expectation $ \Sigma^2+1$. The random tree $ \mathcal{T}_\infty$ is an infinite plane tree, obtained as follows: Start from a semi-infinite line $\{S_{0}, S_{1}, \dots\}$ of vertices rooted at $S_0$, called the \emph{spine}, and graft independently on each $S_{i}$ a random number $Y-1$ of independent $\nu$-Galton--Watson trees, see e.g. \cite{LG05} for a definition,
%\footnote{this probability measure on rooted plane trees gives probability $\prod_{x \in \mathfrak{t}} \nu(\text{deg}(x))$ to the finite rooted tree $\mathfrak{t}$, where $\text{deg}$ is the outdegree of $x$ in $\mathfrak{t}$ (with respect to the root).}
where $Y \sim \overline{\nu}$. To get a plane tree (i.e.~a tree with a planar embedding), independently for each vertex of the spine, consider a random uniform ordering of the children. See \cite{AD14} for details and Figure \ref{fig:kestentree} for an illustration. In particular the mean number of trees grafted on each vertex of the spine is $\sum_{k\geq 1} \overline{\nu}_{k} \cdot (k-1) = \Sigma^{2}$.

\begin{figure}[!h]
 \begin{center}
\includegraphics[width=\linewidth]{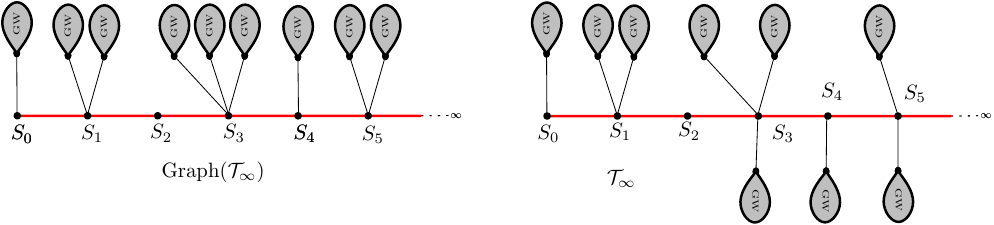}
 \caption{ \label{fig:kestentree}The construction of  $ \mathcal{T}_\infty$ from a spine and a random $\overline{\nu}-1$ number of  $\nu$-Galton--Watson trees grafted on each vertex. After assigning an order to the children of each vertex of the spine we obtain the plane tree $ \mathcal{T}_{\infty}$.}
 \end{center}
 \end{figure}
 
%Clearly, in the above construction, the tree $ \mathcal{T}_\infty$ comes with a distinguished infinite ray  $(S_0,S_1,S_2, \dots)$ corresponding to the genealogical line of the mutant particles. Notice that the number of trees branched to any vertex of this spine has law $\overline{\nu}-1$ and in particular it has mean $ \Sigma^2$.

\label{sec:infty}
The parking process is easy to analyse on $ \mathcal{T}_{\infty}$ as observed in \cite{GP19}. {To do this, we shall perform the parking process in two stages: we first (try to) park all the cars arriving in the subtrees grafted to the spine of $ \mathcal{T}_{\infty}$ and then park the remaining cars arriving on the spine of $ \mathcal{T}_{\infty}$. After performing the first stage, let us focus} on the number of incoming cars at $S_{h}$ the $h$-th vertex on the spine coming from the ``branches on the sides'' {(first stage of the parking process)} together with the possible cars arriving precisely at this vertex. By the description of $ \mathcal{T}_{\infty}$, these quantities are i.i.d.~and distributed according to the law of the random variable $Z$ defined by:
 \begin{eqnarray} Z= \sum_{i=1}^{ Y-1} F_{i} + P,   \label{eq:loiZ}\end{eqnarray}
where $ Y \sim \overline{\nu}$, $F_{i} \sim \varphi_{}( \mathcal{T})$, $P \sim \mu$ are all independent. In particular, the mean of $Z$ is equal to $ \Sigma^2 \mathbb{E}[\varphi( \mathcal{T})]+m$ which is the quantity appearing in Line (ii) of Theorem \ref{thm:main}. {Performing the second stage of the parking}, it is easy to see that the flux at the root can be written as 
\begin{eqnarray}
\varphi_{}( \mathcal{T}_{\infty}) = \sup_{h \geq 0} \big(Z_{0}+ \dots+ Z_{h}- (h+1)\big)\vee 0, 
\label{eq:supZ}
\end{eqnarray} 
where $Z_{i}$ are i.i.d.~copies of $Z$. This settles the case of the local limit easily: when $ \mathbb{E}[Z] \geq 1$ (which corresponds to the supercritical or critical case), the random walk $W$ with i.i.d.~increments with law $Z-1$ oscillates or drifts to $+\infty$\footnote{"oscillates" means $\limsup W_n= -\liminf W_n=+\infty$ whereas "drifts to infinity" means $\lim W_n=+\infty$, both properties holding almost surely} and the flux at the root of $ \mathcal{T}_{\infty}$ is infinite with probability $1$. When $ \mathbb{E}[Z]<1$ (which corresponds to the subcritical case), the random walk $W$ has a negative drift, and $ \varphi( \mathcal{T}_\infty)$ is almost surely finite.

\subsection{Spinal decomposition for $ \mathcal{T}$} If $ \mathfrak{t}$ is a plane tree given with a distinguished vertex $x \in \mathfrak{t}$, we denote by $ \mathrm{Top}( \mathfrak{t},x)$ the subtree of the descendants of $x$ and let $ \mathrm{Pruned}( \mathfrak{t},x)$ be the tree obtained from $ \mathfrak{t}$ by removing $ \mathrm{Top}( \mathfrak{t},x) \backslash \{x\}$, see Figure \ref{fig:prunedtrunk}. 
The spinal decomposition   for a \emph{critical} Galton--Watson tree $\mathcal{T}$  reads as follows: 
 \begin{eqnarray} \mathbb{E}\left[\sum_{x \in \mathcal{T}} F\Big(\mathrm{Pruned}(  \mathcal{T},x) , \mathrm{Top}( \mathcal{T},x) 	\Big)\right] 
  = \sum_{h \geq 0} \mathbb{E}\left[F\Big( \mathrm{Pruned}( \mathcal{T}_{\infty}, S_{h}) , \mathcal{T} 	\Big)\right], \label{eq:spinal}
 \end{eqnarray}
for any positive function $F$,  where in the last expectation $ \mathcal{T}_{\infty}$ and $ \mathcal{T}$ are independent. See \cite[Chapter 12.1]{LP16} 
%or \cite[Eq. (24)]{DU09} 
from which the statement is easily derived. We shall use the straightforward extension of this equation to trees decorated with i.i.d.~labels : in this extension, $\mathcal{T}$ on the LHS and $\mathcal{T}_{\infty}$ and $\mathcal{T}$ on the RHS are replaced by their labelled version, the labels being i.i.d.~random variables with law $\mu$. 

 \subsection{Comparisons between $ \mathcal{T}_n$ and $ \mathcal{T}_\infty$}
 \label{sec:decomposition}
 In the case of $ \mathcal{T}_n$, the spinal decomposition is more intricate but we will only need a rough control. 
 %no exact spinal decomposition  as \eqref{eq:spinal} holds.
%olivier : on peut toujours mettre F=F 1_{|T|=n} ca reste exact mais pas pratique; the use of a spinal decomposition is more intricate ? 
%However, if we restrict to a neighborhood of the spine then we can compare the model $ \mathcal{T}_n$ with $ \mathcal{T}_\infty$ with a good accuracy. 
To define a spine in  $\mathcal{T}_{n}$, conditionally on $ \mathcal{T}_{n}$ we sample a uniform vertex $ \mathcal{V}_{n} \in \mathcal{T}_{n}$. It is standard that the height $| \mathcal{V}_{n}|$ of $ \mathcal{V}_{n}$ converges once renormalized by $ \sqrt{n}$ towards a Rayleigh distribution, more precisely, we have the following local limit law established in \cite[Eq (12)]{KR18}
 \begin{eqnarray} \label{local:height} \sup_{ \varepsilon < t < 1/ \varepsilon} \sqrt{n} \left|\mathbb{P}\left( | \mathcal{V}_{n}| = \lfloor t\frac{ \sqrt{2n}}{ \Sigma}\rfloor \right) -  \frac{\Sigma}{ \sqrt{2n}}  \cdot 2t  \mathrm{e}^{-t^{2}/2}\right|  \xrightarrow[n\to\infty]{}0,  \end{eqnarray} for any $ \varepsilon>0$. To get a control on large parts of the tree $ \mathcal{T}_{n}$, we shall decompose it into three pieces. Recall the definition of $ \mathrm{Pruned}( \mathfrak{t},x)$ for a plane tree $ \mathfrak{t}$ with a distinguished vertex $x$. We shall further decompose $ \mathrm{Pruned}( \mathfrak{t},x)$ into two rooted plane trees carrying a distinguished vertex by considering $ \mathrm{Down}( \mathfrak{t},x) = \mathrm{Pruned}( \mathfrak{t}, y)$ where $y$ is the ancestor of $x$ at height\footnote{the height $|x|$ of a vertex $x$ in a rooted tree is defined as the number of edges along the unique non-intersecting path between that vertex and the root.} $\lfloor |x|/2\rfloor$. 
 We also set  $ \mathrm{Up}( \mathfrak{t},x) = \mathrm{Top}(\mathrm{Pruned}( \mathfrak{t},x), y)$.  See Figure \ref{fig:prunedtrunk} for an illustration. Each of these three trees is rooted at the unique vertex of its vertex set that has minimal height: $\mathrm{Top}( \mathfrak{t},x)$ is rooted at $x$,  $\mathrm{Up}( \mathfrak{t},x)$ at $y$ and $\mathrm{Down}( \mathfrak{t},x)$  at the root $\varnothing$ of $\mathfrak{t}$.
  \begin{figure}[!h]
  \begin{center}
  \includegraphics[width=0.7\linewidth]{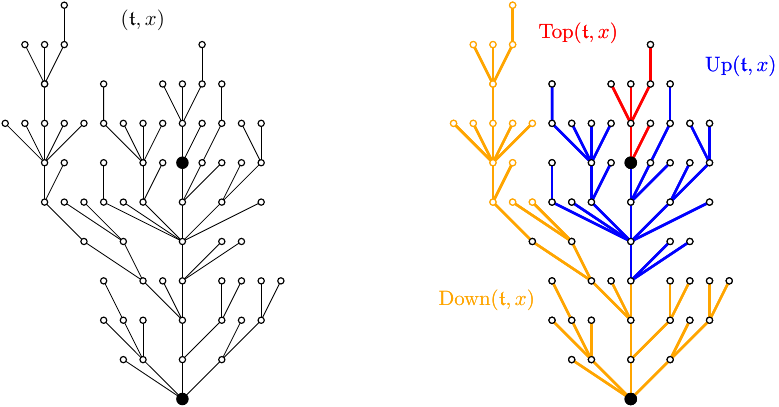}
  \caption{ \label{fig:prunedtrunk}The different pieces cut out of a plane tree $ \mathfrak{t}$ with a distinguished vertex $x$.}
  \end{center}
  \end{figure}
  
\begin{lemma}[Rough control] \label{lem:rough} Conditionally on $ \mathcal{T}_n$, let $ \mathcal{V}_n$ be a uniform vertex in $ \mathcal{T}_n$ whose height is denoted by  $H_{n}$ and let $ \mathcal{T}_{\infty}$ be independent of $( \mathcal{T}_{n}, \mathcal{V}_{n})$. % If $ \mathcal{T}_{\infty}$ is independent of $ (\mathcal{T}_{n}, \mathcal{V}_{n})$ then
   For every $ \varepsilon>0$, there exists $\delta >0$ and $n_{0} \geq 0$ such that for all $n \geq n_{0}$ and any event $A$
  $$ \mathbb{P}\left(\mathrm{Down}(  \mathcal{T}_{\infty}, S_{H_{n}}) \in A \right) \leq \delta \quad  \Longrightarrow \quad 
   \mathbb{P}\left(\mathrm{Down}(  \mathcal{T}_{n},  \mathcal{V}_{n}) \in A \right) \leq \varepsilon.$$
and:
$$ \mathbb{P}\left(\mathrm{Up}(\mathcal{T}_{\infty}, S_{H_{n}}) \in A \right) \leq \delta \quad  \Longrightarrow \quad \mathbb{P}\left(\mathrm{Up}(  \mathcal{T}_{n},  \mathcal{V}_{n}) \in A \right) \leq \varepsilon.$$
 \end{lemma}
 \proof Fix a  tree $ \mathfrak{t}_0^{\bullet}$ with $n_{0}\leq n$ vertices and a distinguished leaf at height $\lfloor h/2\rfloor$. 
We claim that: 
\begin{equation}
\label{eq:Down}
\frac{\mathbb{P}(\mathrm{Down}(  \mathcal{T}_{n},  \mathcal{V}_{n})  = \mathfrak{t}_0^{\bullet} \mid | \mathcal{V}_{n}|=h)}{\mathbb{P}(\mathrm{Down}(  \mathcal{T}_{\infty},  S_{h})  = \mathfrak{t}_0^{\bullet})} = \frac{ (n-n_{0}+1) \cdot \mathbb{P}(| \mathcal{T}|=n-n_{0}+1) \mathbb{P}(| \mathcal{V}_{n-n_{0}+1}|=h-\lfloor h/2\rfloor)}{n \cdot \mathbb{P}(| \mathcal{T}|=n) \mathbb{P}(| \mathcal{V}_{n}|=h)}\cdot
\end{equation}
 To get the claim \eqref{eq:Down}, notice that for any non-negative measurable function $F$ on the set of rooted planar pointed trees, denoting by $\mathrm{GW}$ the $\nu$-Galton-Watson measure, given by $\mathrm{GW}(\mathfrak{t}) =\prod_{x \in \mathfrak{t}} \nu(\{\text{deg}(x)\})$ for a finite rooted tree $\mathfrak{t}$, where $\text{deg}$ is the outdegree  (with respect to the root) of $x$ in $\mathfrak{t}$, we have the equality
$$ \mathbb{E}[F(\mathcal{T}_{n},  \mathcal{V}_{n})] = \frac{1}{n} \frac{\sum_{v\in \mathfrak{t}} \Ind{|\mathfrak{t}|=n} F(\mathfrak{t},v) \mathrm{GW}(\mathfrak{t}) }{\mathbb{P}(|\mathcal T|=n)},$$   
 hence, conditioning further by the height of a random vertex $\mathcal{V}_{n}$, we obtain :
$$ \mathbb{E}[F(\mathcal{T}_{n},  \mathcal{V}_{n}) \mid |\mathcal{V}_{n}|=h ] = \frac{1}{n} \frac{\sum_{v\in \mathfrak{t}} \Ind{|\mathfrak{t}|=n, |v|=h} F(\mathfrak{t},v) \mathrm{GW}(\mathfrak{t}) }{\mathbb{P}(|\mathcal T|=n) \mathbb{P}(| \mathcal{V}_{n}|=h) }.$$   
Now, taking for $F$ the function $F=\Ind{\mathrm{Down}(\mathfrak{t},v)=\mathfrak{t}_0^{\bullet}}$, and using that a tree $\mathfrak{t}$ with $\mathrm{Down}(\mathfrak{t},v)= \mathfrak{t}_0^{\bullet}$ is the concatenation of $\mathfrak{t}_0^{\bullet}$ and a tree  $\mathfrak{t}_1$ grafted on it, and that the $\mathrm{GW}$ measure\footnote{We slightly abuse notation by writing 
$\mathrm{GW}(\mathfrak{t}_0^{\bullet})$ for the Galton-Watson measure of the corresponding unpointed tree.} of $\mathfrak{t}$ splits in a simple way, we obtain that the sum at the numerator takes the following form:
\begin{align*}
\sum_{v\in t} \Ind{|\mathfrak{t}|=n, |v|=h} F(\mathfrak{t},v) \mathrm{GW}(\mathfrak{t})
& =\sum_{\mathfrak{t}_1, v_1 \in \mathfrak{t}_1}  \Ind{|\mathfrak{t}_1|=n-n_0+1, |v_1|=h-\lfloor h/2\rfloor} \frac{\mathrm{GW}(\mathfrak{t}_0^{\bullet})}{\nu(\{0\})} \mathrm{GW}(\mathfrak{t}_1) \\
%& =\frac{\mathrm{GW}(\mathfrak{t}_0^{\bullet})}{\nu(\{0\})}  \sum_{\mathfrak{t_1}, v_1\in \mathfrak{t_1}}  \Ind{|\mathfrak{t}_1|=n-n_0+1, |v_1|=h-\lfloor h/2\rfloor}  \mathrm{GW}(\mathfrak{t}_1)\\
& =\frac{\mathrm{GW}(\mathfrak{t}_0^{\bullet}) }{\nu(\{0\})} (n-n_0+1) \mathbb{P}(|\mathcal T|=n-n_0+1) \mathbb{P}(|\mathcal V_{n-n_0+1}|=h-\lfloor h/2\rfloor)
\end{align*}  
whereas  by construction of the Kesten's tree, the denominator simply equals:
$$\mathbb{P}(\mathrm{Down}(  \mathcal{T}_{\infty},  S_{h})  = \mathfrak{t}_0^{\bullet}) = \frac{\mathrm{GW}(\mathfrak{t}_0^{\bullet}) }{\nu(\{0\})},$$
%$$\mathbb{P}(\mathrm{Down}(  \mathcal{T}_{\infty},  S_{h})  = \mathfrak{t}_0^{\bullet}) = \mathrm{GW}(\mathfrak{t}_0^{\bullet}) ,$$
see equation (12.1) in \cite{LP16},
which cancels out the same term at the numerator, giving the claim \eqref{eq:Down}.
Using \eqref{local:height} now, the right hand side of \eqref{eq:Down} is bounded by some absolute constant $C_{\alpha}>0$ as long as $ \alpha \sqrt{n} \leq h \leq \alpha^{-1} \sqrt{n}$ and $n_{0} \leq (1- \alpha) n$. We deduce that for any event $A$ we have 
 \begin{eqnarray*} \mathbb{P}(\mathrm{Down}(  \mathcal{T}_{n},  \mathcal{V}_{n}) \in A) &\leq& C_{\alpha}   \mathbb{P}(\mathrm{Down}(  \mathcal{T}_{\infty},  S_{H_{n}}) \in A)\\ &&  + \mathbb{P}( H_{n}/\sqrt{n} \notin [ \alpha, \alpha^{-1}]  \mbox{ \ \ or \ \ } |\mathrm{Down}(  \mathcal{T}_{n},  \mathcal{V}_{n})| \geq (1- \alpha)n).  \end{eqnarray*}
Using standard scaling limit results for $ (\mathcal{T}_{n}, \mathcal{V}_{n})$,  
%\cite[Theorem 1.6]{JA05} for $H_n$, 
the second probability in the right-hand side can be made smaller than $\varepsilon/2$  (for all $n$ large enough) by choosing $\alpha>0$ small enough. Putting $\delta =\varepsilon/(2 C_{ \alpha})$ we indeed deduce that $\mathbb{P}(\mathrm{Down}(  \mathcal{T}_{\infty},  S_{H_{n}}) \in A) \leq \delta$ implies $\mathbb{P}(\mathrm{Down}(  \mathcal{T}_{n},  \mathcal{V}_{n}) \in A)  \leq \varepsilon$ as desired.
The result for the $ \mathrm{Up}$ part can be deduced by symmetry. 
\qed

 \subsection{Fringe trees and a law of large numbers for the flux}
Given a plane tree $ \mathfrak{t}$, the fringe subtree distribution is the empirical measure 
$$ \mathrm{Fringe}(  \mathfrak{t}) =  \frac{1}{\# \mathfrak{t}} \sum_{x \in \mathfrak{t}} \delta_{ \mathrm{Top}(  \mathfrak{t},x)}.$$
A result of Janson  \cite[Theorem 1.3, Quenched version, Formula (1.11)]{JA16} states that $\mathrm{Fringe}( \mathcal{T}_{n})$ converges in probability (for the total variation distance) towards the distribution of the $\nu$-Galton--Watson measure. For our purposes, the definition of $ \mathrm{Fringe}$ is easily extended by taking care of the labeling $\ell : \mathcal{T}_{n} \to \mathbb{Z}_{\geq 0}$ and enables us to establish an ``abstract'' law of large numbers for the flux $\varphi( \mathcal{T}_{n})$:

\begin{lemma}(Weak law of large numbers for the flux in conditioned Galton--Watson trees)\label{prop:janson-fringe}. Recall that $m$ is the mean of $\mu$. 
The flux at the root of $ \mathcal{T}_{n}$ satisfies
\begin{eqnarray} \frac{\varphi_{}( \mathcal{T}_n)}{n} \quad \xrightarrow[n\to\infty]{( \mathbb{P})}  \quad m- \mathbb{P}(\mbox{a car is parked at } \varnothing \mbox{ in }  \mathcal{T} \mbox{ after parking}).  
\label{fringe}
\end{eqnarray}
\end{lemma}

\proof 

For a rooted labelled tree $\mathfrak t$, let $E(\mathfrak{t})$ be the event that a car is parked at the root of $\mathfrak{t}$ after parking, so that the quantity $\mathbb{P}(\mbox{a car is parked at } \varnothing \mbox{ in }  \mathcal{T} \mbox{ after parking})$ on the right-hand side of \eqref{fringe} corresponds to $\mathbb{P}( E(\mathcal T))$. Recall that  conditionally on $ \mathcal{T}_{n}$, the car arrivals $(L_x: x \in \mathcal{T}_{n})$ are i.i.d.~with law $\mu$. By the conservation of cars\footnote{This argument does not work in France on New Year's Eve where about $1000$ cars are burned.} we have  
\begin{align*} 
\frac{1}{n}
\left(\sum_{x \in \mathcal{T}_n} L_x  - \varphi( \mathcal{T}_{n})  \right)
 & = \frac{1}{n}  \sum_{x \in \mathcal{T}_n} \Ind{ E(\mathrm{Top}( \mathcal{T}_{n},x))} \\
&  = \int \Ind{E(\mathfrak{t})} \;  d \mathrm{Fringe} (\mathfrak{t}), \end{align*} 
and the convergence in probability of the Fringe probability measure entails that:
$$\int \Ind{E(\mathfrak{t})} \;  d \mathrm{Fringe} (\mathfrak{t}) \quad \xrightarrow[n\to\infty]{( \mathbb{P})}\quad  \mathbb{P}( E(\mathcal T)).$$
Since $\frac{1}{n}\sum_{x \in \mathcal{T}_n} L_x \to m$ in probability by the  law of large numbers, the desired result follows. \endproof

\section{$ \mathbb{E}[\varphi( \mathcal{T})]$ via spine decomposition and a differential equation}
In this section we compute $ \mathbb{E}[\varphi( \mathcal{T})]$ thus proving line $(ii)$ in Theorem \ref{thm:main}. This is done using a differential equation (more precisely its integral version) obtained, roughly speaking, by letting the cars arrive one-by-one and computing the marginal contribution to the flux using the spine decomposition \eqref{eq:spinal}. The same method is applied to estimate $\mathbb{P}( \varnothing \mbox{ contains a car in } \mathcal{T})$ and yields line $(iii)$ of Theorem \ref{thm:main} which in turn implies parts of line $(i)$ by Lemma \ref{prop:janson-fringe}.

\subsection{The mean flux $ \mathbb{E}[\varphi( \mathcal{T})]$}
Conditionally on $\mathcal T$, we define $( A_x, L_x)_{x \in \mathcal{T}}$ a collection of independent random variables distributed as $ \mathrm{Unif}[0,1] \otimes \mu$. The variable $ {A}_{x}$ will be thought of as ``the time of arrival'' of the $ {L}_{x}$ cars on the vertex $x$. This enables us to define an increasing labeling $L^{(t)} : \mathcal{T} \to \{0,1,2, \dots \}$ by setting 
 $$ L^{(t)} (x)= \mathbf{1}_{ {A}_{x} \leq t} \cdot {L}_{x} \quad \mbox{ with associated flux } \quad \varphi(t) := \varphi( \mathcal{T}, L^{(t)}).$$
 %olivier : L_x puis L^{(t)}(x) : mettre les deux en indice ou en arguments.
 Obviously, $ L^{(t)}$ is an i.i.d.~labeling of $ \mathcal{T}$ with law $\mu_t= (1-t) \delta_0 + t \mu$ with mean  $mt$. % and variance $\sigma^2 t +m^2 (t-t^2)$.   
In the following, we take profit of the arrival times $A_{x}$ to park the cars sequentially (which is allowed by the Abelian property of the model).  
 \begin{proposition}[Phase transition for the mean flux] \label{prop:meanflux} For $t \in [0,1]$ let  $  \varPhi(t) = \mathbb{E}[\varphi(t)]$ be the mean flux in $\mathcal{T}$ with car arrivals with law $\mu_t$. If $t_{\max}$ is the smallest positive solution of the equation $(1-mt)^2=t\Sigma^2 (\sigma^2 +m^2-m)$ (set $t_{\max}=+\infty$ in case no such solution exists), then 
  \begin{eqnarray} \varPhi(t) = \left\{\begin{array}{lcc}
  \displaystyle\frac{(1-mt) - \sqrt{(1-mt)^2- \Sigma^2 (\sigma^2 +m^2-m) t }}{\Sigma^2} & \mbox{ if }& t \leq t_{\max}\\
 +\infty & \mbox{ if } & t> t_{\max}. \end{array}\right. \label{eq:f(t)}  \end{eqnarray}
 \end{proposition}

\begin{proof}[Proof of Line $(ii)$ of Theorem \ref{thm:main}]
%Line $(ii)$ of Theorem \ref{thm:main} follows, noting
It remains to justify the following alternative characterization of the three regimes described in Theorem \ref{thm:main} by mean of the parameter  $t_{\max} \in ]0,+\infty]$ : $t_{\max}<1$ in the supercritical regime $\Theta<0$, $t_{\max} =1$ in the critical regime $\Theta=0$ and $t_{\max}>1$ in the subcritical regime $\Theta>0$. To check these claims, observe that the function 
$$t \mapsto (1-mt)^2- t\Sigma^2 (\sigma^2 +m^2-m)$$
is decreasing on $[0,1/m]$ (hence on $[0,1]$ since we assumed $m\le 1$), as the sum of a decreasing function on $[0;1/m]$, $t \mapsto (1-mt)^2$, and a non increasing-one on $\R$ (the coefficient $\sigma^2 +m^2-m = \E[L(L-1)]$ is non-negative). \end{proof}

\begin{proof}[Proof of Proposition \ref{prop:meanflux} ]
Let $t \in [0,1]$ and write 
 \begin{align*}
\varPhi(t)& = \E\Big[ \sum_{x \in \mathcal{T}} I^{x}(A_x) \Ind{0 \leq A_x \leq t } \Big]  \end{align*}
 where $I^{x}(s)$ is the number $\in \{0,1, \dots , L_{x}\}$  of cars that arrived at time $s$ on the vertex $x$ which contribute to $\varphi(t)$, i.e.~those that did not manage to park at their arrival time $s \leq t$. Integrating over the value $s=A_x$ and using the spinal decomposition \eqref{eq:spinal} -- more precisely its easy extension to decorated Galton--Watson trees -- we can write the previous display as 
 \begin{align}
\label{eq:spinal1} \varPhi(t) &  =    
\int_{0}^{t} \mathrm{d}s \ \E\Big[ \sum_{x \in \mathcal{T}} I^{x}(s) \Big]  \underset{ \eqref{eq:spinal}}{=}   \int_{0}^{t} \mathrm{d}s \ \sum_{h =0}^{\infty}\mathbb{E}[ I(s,h)]  \end{align}
 where $I(s,h)$ is obtained as follows: For $h \geq 0$ define a tree $ \mathcal{T}(h)$ by grafting an independent copy of $ \mathcal{T}$ on top of $ \mathrm{Pruned}( \mathcal{T}_{\infty}, S_{h})$. This tree is decorated by letting i.i.d.~car arrivals with law $\mu_{s}$ \emph{except} on the vertex $S_{h}$ where we put an independent number of cars distributed as $\mu$. Then $I(s,h)$ is the number of those cars arriving on $ S_{h}$ that do not manage to park after having parked all other cars of $ \mathcal{T}(h)$. See Figure~\ref{fig:explained} for an illustration. 
 
 \begin{figure}[!h]
  \begin{center}
  \includegraphics[width=0.7\linewidth]{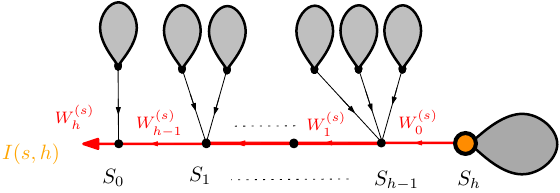}
  \caption{ \label{fig:explained}Definition of the variable $ I(s,h)$: add on top of $ \mathrm{Pruned}(  \mathcal{T}_{\infty}, S_{h})$ an independent unconditioned Galton--Watson $ \mathcal{T}$. Label all the vertices except $S_{h}$ of the resulting tree with i.i.d.~car arrivals with law $\mu_{s}$ and let them park. The variable $I(s,h)$ is then the number of cars among an independent number $\sim\mu$ of cars that we add on $S_{h}$ that do not manage to park.}
  \end{center}
  \end{figure}
 To compute $\mathbb{E}[ I(s,h)]$ we proceed as in Section \ref{sec:infty} and notice that at time $s_{-}$ the collection of the outgoing fluxes from the vertices $S_{h-i}$ before parking the cars arriving on $S_{h}$ defines when $i$ runs though the set $\{0, \ldots, h\}$ a random walk $( W^{(s)}_{i} : 0 \leq i \leq h)$ of length $h$ with i.i.d.~increments with law $Z^{(s)}-1$, where $Z^{(s)}$ is defined as in \eqref{eq:loiZ} by
 \begin{eqnarray*} Z^{(s)}= \sum_{i=1}^{ Y-1} F^{(s)}_{i} + P^{(s)},  \end{eqnarray*}
where $ Y \sim \overline{\nu}$, $F_{i} \sim \varphi_{}( \mathcal{T}, L^{(s)})$, $P \sim \mu_{s}$ are all independent. Besides, the starting point $W_{0}^{{(s)}}$  is distributed as the sum of $\nu$ independent copies of a random variable with law $ \varphi( \mathcal{T}, L^{(s)})$ (and is independent of the increments of the walk) minus $1$. Write $T^{(s)}_{-i}$ for the hitting time of $-i \leq 0$ by this left-continuous random walk. Assume all the cars whose arrival vertex is distinct from $S_h$ have been parked, and consider the $i$-th car arrived on vertex $S_h$. This car contributes to the flow at the root iff $\{T_{-i}^{(s)}>h\}$. Summing over the $L \sim \mu$ cars that arrive at vertex $S_h$, we find the representation: $I(s,h)=\sum_{i=1}^L \Ind{\{T_{-i}^{(s)}>h\}}$. Hence, writing $  \mathbb{P}_{x}$ for the law of the walk $(W_i^{(s)})_{i \geq 0}$ started at $x$, we obtain:
$$
 \mathbb{E}[I(s,h)]  =  \mathbb{E}\left[ \sum_{i=1}^{L} \P_{W_{0}^{{(s)}}}(T^{(s)}_{-i} > h)\right],$$
 where $L \sim \mu$ is independent of the walk $W^{(s)}$.  Performing the sum on $h$ we get 

\begin{align}
\sum_{h=0}^{\infty}\mathbb{E}[I(s,h)] 
& = \sum_{h =0}^{\infty} \E\Big[ \sum_{i=1}^{L} \P_{W_0^{(s)}}(T_{-i}^{(s)} > h)\Big]  \nonumber \\
& = \E\Big[ \sum_{i=1}^{L} \E_{W_0^{(s)}}[T_{-i}^{(s)}]\Big] \nonumber \\
& = \E\Big[ \sum_{i=1}^{L} (\E[W_{0}^{(s)}]+i) \;\E_{0}[T_{-1}^{(s)}]\Big]\nonumber \\
& = \left( \E[L]\E[W_{0}^{(s)}] +\E\Big[\frac{L(L+1)}{2}\Big]\right) \E_0[T_{-1}^{(s)}]   \nonumber \\ 
& = \left( m \varPhi(s) + \frac{1}{2}(\sigma^2+m^2-m)\right) \E_0[T_{-1}^{(s)}]. \label{eq:wald1}
\end{align}

Furthermore, if $ \mathbb{E}[Z^{(s)}-1] \geq 0$, then the random walk $W^{(s)}$ has a positive (or zero) drift, so $ \E_0[T_{-1}^{(s)}]= \infty$. On the other hand, when $ \mathbb{E}[Z^{(s)}-1] < 0$ i.e.~if $1-ms -\Sigma^2 \varPhi(s)  > 0$ then the random walk $W^{(s)}$ has a strictly negative drift, and an application of Wald's lemma gives: 
$$\E_0[T^{(s)}_{-1}]=  \frac{1}{ \mathbb{E}[1-Z^{(s)}]}=\frac{1}{1-ms -\Sigma^2 \varPhi(s) }.$$
{Combining the previous displays, since $\varPhi$ is left-continuous (by monotone convergence),  we deduce that  $\varPhi$ satisfies the  integral equation:
$$ (\star)\qquad \varPhi(0) =0, \quad \varPhi(t) = \int_0^t \mathrm{d}s \left(\frac{ \frac{1}{2}(\sigma^2+m^2-m)+ m \varPhi(s)   }{1-ms -\Sigma^2 \varPhi(s)  }\right), \quad 0 \leq t \leq t_c$$ 
where $t_c = \inf\{ t \in [0,1] : 1-mt -\Sigma^2 \varPhi(t)  < 0\}$, and  $ \varPhi(t) = \infty$ for all $ t_{c} < t \leq 1$. It is easy\footnote{To derive the solution to the differential equation $(\star)$, observe that the derivative of the function from $[0,t_c)$ to $\R^+$ that maps $t$ to $(1-mt- \Sigma^2 \varPhi(t))^2$ is an affine function of $t$.}  to check that the function defined on the right-hand side of \eqref{eq:f(t)}, call it $f(t)$ for the time being, is a solution to $(\star)$ with $t_c=t_ \mathrm{max}$.\\
We will first prove that $\varPhi(t) \geq f(t)$ for all $t \in [0,1]$. To see this, notice that $\varPhi = f$ on $[0, t_c \wedge t_ \mathrm{max})$ since they satisfy the same well-posed differential equation. This also holds at $t_c\wedge t_{\max}$ by left-continuity of $\varPhi$ and $f$.  Since $\varPhi(t) = \infty$ for $t>t_{c}$ the statement follows.} 

We now prove that $\varPhi(t) \leq f(t)$ for all $t \in [0,1]$. The problem comes from the fact that $\varPhi$ may coincide with $f$ for small $t$ and then decide to ``explode'' to $+\infty$ at some $t_c < t_{ \mathrm{max}}$ before $f$ does so (this procedure in fact defines a family of solutions to $(\star)$ indexed by their jump time to $\infty$). To show that this cannot happen, we introduce $\varphi_{n}(t)$ the flux at the root for the tree $\mathcal T$ decorated with $\mu_t$-arrivals restricted to those vertices at distance at most $n$ from the root. We write $ \varPhi_n(t) = \mathbb{E}[\varphi_n(t)]$. Clearly by monotone convergence we have $\varPhi_n(t) \uparrow \varPhi(t)$ as $n \to \infty$ for any $t \geq 0$. We also have the  bound  $$ \varPhi_n(t) \leq  \mathbb{E}\Big[\sum_{ \begin{subarray}{c}	
x \in \mathcal{T}\\ |x|\leq n \end{subarray}} L^{(t)}(x)\Big] = m  t \cdot  \mathbb{E}\Big[\sum_{ \begin{subarray}{c}	
x \in \mathcal{T}\\ |x|\leq n \end{subarray}} 1\Big] =  m t n,$$
and dominated convergence ensures the map $ t  \mapsto \varPhi_n(t)$ is continuous on $\R^+$. Using the monotonicity of the parking process with respect to the labeling, one can repeat the argument yielding to \eqref{eq:spinal1} and \eqref{eq:wald1} and we claim that we get this time the inequality 
%justify inequality
$$\varPhi_{n}(t) \leq \int_{0}^{t}  \mathrm{d}s\,   \frac{ \frac{1}{2}(\sigma^2+m^2-m)+ m \varPhi_{n}(s)   }{1-ms -\Sigma^2 \varPhi_{n}(s)  }, $$
valid as long as $(1-mt) -\Sigma^2 \varPhi_{n}(t) \geq 0$: 
to wit, notice that one can represent the function $\varPhi_n$ as 
$ \varPhi_n(t)   =      \int_{0}^{t} \mathrm{d}s \ \sum_{h =0}^{\infty}\mathbb{E}[ I_n(s,h)],$
where $I_n(s,h)$ is the number of those cars counted in $I(s,h)$ whose arrival vertex is at distance at most $n$ from the root of $ \mathcal{T}(h)$. But $I_n(s,h)$ is in turn bounded by  the number of those cars  counted in $I(s,h)$ whose arrival vertex is at distance at most $n$ from the set of vertices $\{S_0, \ldots, S_h\}$ of  $\mathcal{T}(h)$ (the so-called spine), see figure \ref{fig:explained}. 
Replacing $\varPhi$ by $\varPhi_n$, the rest of the equalities leading to  $(\star)$  still hold.

Using the continuity of $\varPhi_n$ and the previous display, it is an easy exercise to show that for every $n$ we have $ \varPhi_n \leq f$ on $[0, t_{ \max}]$ (including $t_{\max}$). Sending $n \to \infty$, we deduce that $\varPhi \leq f$ on $[0, t_{ \max}]$ as desired. \end{proof}

 \subsection{The probability the root of a Galton--Watson tree is parked}
 Recall the characterization of the phases using $t_{ \mathrm{max}}$ or $\Theta$. In the next proposition we control the probability, under $ \mathcal{T}$, that the root vertex contains a car.  This gives Line $(iii)$ of Theorem \ref{thm:main} and combined with Lemma \ref{prop:janson-fringe} shows that the flux is linear in the supercritical regime (Line $(i)$ right in Theorem \ref{thm:main}) and sublinear in the critical regime (Line $(i)$ middle). 
 \begin{proposition}\label{conj-GP} With the same notation as in Proposition \ref{prop:meanflux} we have:
\begin{eqnarray*}
 \mathbb{P}( \varnothing \mbox{ is parked in } \mathcal{T}) \left\{\begin{array}{ccl} \displaystyle =m & \mbox{ in the critical or subcritical case,} \\
 <m & \mbox{ in the supercritical case}.  \end{array} \right.
  \end{eqnarray*}
 \end{proposition}
 \begin{proof} We use the same notation and  proceed as in the proof of Proposition \ref{prop:meanflux} where the cars arrive according to random times $A_{x}$ on the tree $ \mathcal{T}$. Putting $ p_{t}= \mathbb{P}(\varnothing \mbox{ contains a car  in } (\mathcal{T}, L^{(t)}))$ we have using the spine decomposition 
 \begin{align*}
p_{t} = \int_{0}^{t}  \mathrm{d}s\  \sum_{h \geq 0} \mathbb{P}( P(s,h)), \end{align*}
where $P(s,h)$ is the event that in the labeled tree described in Figure \ref{fig:explained}, one of the cars arriving on the vertex $S_{h}$ at time $s$ goes down the spine and manages to park on the empty root vertex $\varnothing$. With the same notation as in the display after Figure \ref{fig:explained} we have 
$ P(s,h) =   \bigcup_{i=1}^{L} \{T_{-i}^{(s)}=h\}$ under $\mathbb{P}_{W_{0}^{(s)}}$ and so performing the sum over $h$ we deduce
\begin{align*}
\sum_{h=0}^{\infty}\mathbb{P}( P(s,h)) &= \E\Big[\sum_{i=1}^{L}  \P_{W_{0}^{(s)}}(T^{(s)}_{-i}<\infty) \Big]  = \E\Big[\sum_{i=1}^{L}  \P_{0} \big(T^{(s)}_{-1}<\infty\big)^{W_{0}^{(s)}+i} \Big] \\  &= \left\{ \begin{array}{ccc} \mathbb{E}[L] =m & \mbox{ if }& s \leq t_{\max}\\
< m & \mbox{ if }& s >t_{\max}, \end{array} \right.
\end{align*}
where $t_{\max}$ is as in Proposition \ref{prop:meanflux}. The proposition follows by integration.
\end{proof}

\section{Remaining proofs}
We now perform the remaining proofs required for Theorem \ref{thm:main}, namely establishing that $\varphi( \mathcal{T}_n)$ converges in law in the subcritical case and diverges in the critical case, as it does in the infinite model $\mathcal T_{\infty}$.
Even though the tree $\mathcal T_{\infty}$ is the local limit of $\mathcal T_{n}$, the flux is not continuous in the local topology and so transposing the properties from one model to the other requires some extra care. \medskip

Since $ \mathcal{T}_n \to \mathcal{T}_\infty$ in distribution in the local sense as $n \to \infty$, we will suppose in this 
section by Skorokhod embedding theorem that this convergence holds almost surely, also taking into account the i.i.d.~car arrivals on those trees. We will show that 
 \begin{eqnarray} \label{eq:goalsub} \varphi(\mathcal{T}_n) \xrightarrow[n\to\infty]{a.s.} \varphi( \mathcal{T}_\infty) \end{eqnarray} which combined with the results of Section \ref{sec:infty} finishes the proof of Theorem \ref{thm:main}.

\paragraph{Critical and supercritical cases.}  A moment's thought shows that we always have
 \begin{eqnarray}\label{eq:fatou} \liminf_{n \to \infty}\varphi ( \mathcal{T}_n)  \geq \varphi( \mathcal{T}_\infty),   \quad a.s.\end{eqnarray}
but that the inequality may be strict\footnote{Consider e.g.~a line segment of length $n$ with $2n$ cars arriving on top, hence a flux $n$ at the root. This converges towards the empty half-line with zero flux.}. Anyway, in the critical and supercritical case since $ \varphi( \mathcal{T}_\infty) = \infty$ by Section \ref{sec:infty}, we always have \eqref{eq:goalsub} as desired.

\paragraph{Subcritical case.} We suppose here that we are in the subcritical regime. The convergence \eqref{eq:goalsub} is granted provided that we can show that the parking process is local, i.e.~that no car contributing to $ \varphi( \mathcal{T}_n)$ comes from far away. To this end, let $ \mathcal{V}_n$ be a uniform vertex of $ \mathcal{T}_n$ and for $M \geq 1$ denote the event 
$$ \mathrm{Good}(M) = \{\exists x \mbox{ ancestor of } \mathcal{V}_n \mbox{ at height }\leq M \mbox{ which contains no car after parking}\}.$$

\begin{lemma}[Locality of the parking process in the subcritical phase] Suppose that $(\mu, \nu)$ is subcritical. For any $ \varepsilon>0$, we can find $M_0$ so that for all $M \geq M_{0}$ and all $n$ large enough we have
$$   \mathbb{P}( \mathcal{T}_n \in  \mathrm{Good}(M)) \geq 1- \varepsilon.$$
\end{lemma}
\begin{proof} Fix $ \varepsilon>0$. Recall the decomposition of $  \mathcal{T}_n$ into three parts $ \mathrm{Top}(  \mathcal{T}_{n},  \mathcal{V}_{n})$, $\mathrm{Up}(  \mathcal{T}_{n},  \mathcal{V}_{n})$ and $\mathrm{Down}(\mathcal{T}_{n},  \mathcal{V}_{n})$ from Section \ref{sec:decomposition}. We denote these parts labeled by their associated car arrivals (on the vertices common to two parts, we duplicate the car arrivals) by $ \mathrm{Top}_n, \mathrm{Up}_n$ and $ \mathrm{Down}_n$ to simplify notation. We claim that the event $ \mathrm{Good}(M)$ happens for $ \mathcal{T}_n$ if {after proceeding to the parking separately in each part we have}
\begin{itemize}
\item The flux at the root of $ \mathrm{Top}_n$ is less than $M$,
\item The flux at the root of $ \mathrm{Up}_n$ is less than $M$,
\item There are more than $3M$ empty spots on the ``spine" of $\mathrm{Down}_n$ and at least one of this spot is at height less than $M$.
\end{itemize}

\begin{figure}[!h]
 \begin{center}
 \includegraphics[width=14cm]{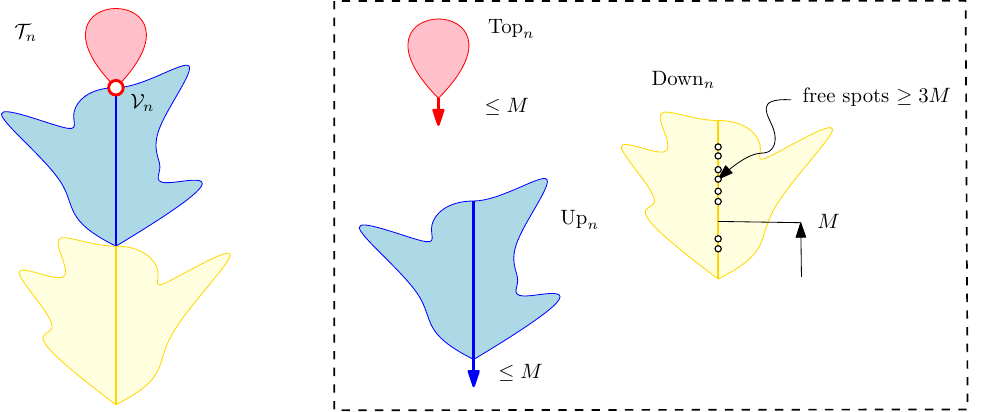}
 \caption{On the right, the three events ensuring $\mathrm{Good}(M)$ happens are realized for each of the three parts $ \mathrm{Top}_n$, $ \mathrm{Up}_n$ and $ \mathrm{Down}_n$. After gluing back the pieces together and parking, the flux coming down from $ \mathrm{Top}_n$ and $ \mathrm{Up}_n$ will be absorbed by the empty spots on the spine of $ \mathrm{Down}_n$, and an empty spot will remain on the spine of $ \mathrm{Down}_n$ at height less than $M$.}
 \end{center}
 \end{figure}

We now use our controls separately on each part to ensure that the \emph{complementary} of each of the previous three events has probability at most $\varepsilon/3$ when $M$ is large enough. 
For $ \mathrm{Top}_n$. This tree converges in distribution towards an (a.s. finite) unconditioned labeled $\nu$-Galton--Watson tree $\mathcal{T}$, \cite[Theorem 1.3]{JA16},
so defining $A=A_M= \{ \varphi > M \}$, we may choose $M$ large enough so that 
$\P(\mathcal{T} \in A_M) < \varepsilon/3$, which implies that for $n$ large enough, $\P(\mathrm{Top}_{n} \in A_M) \leq \varepsilon /3$. 
For $\mathrm{Up}_n$. By subcriticality, the flux is bounded in $\mathcal T_{\infty}$, see Section  \ref{sec:infty}, hence again there is $M$ large enough so that for $n$ large enough, if $H_n$ denotes the height of the vertex $\mathcal{V}_n$, $\mathbb{P}\left(
\mathrm{Up}(  \mathcal{T}_{\infty}, S_{H_{n}}) \in A_M \right) \leq \delta$ for the value of $\delta$ linked with $\varepsilon/3$ in Lemma \ref{lem:rough}, and therefore, for the same values of $n$,  we have $\mathbb{P}\left(\mathrm{Up}_n \in A_M \right) \leq \varepsilon/3$. 
For  $\mathrm{Down}_n$. It follows from Section \ref{sec:infty} (using the same notation) that in $ \mathcal{T}_{\infty}$, the $h$-th vertex on the spine $S_{h}$ is a free spot after parking if and only if we have 
$$ \sup_{i \geq 0} (Z_{h}+ \dots + Z_{h+i} - (i+1)) = -1.$$
Since the random walk with increments $(Z-1)$ has a strictly negative drift in the subcritical case, it follows from standard consideration on random walks  and the fact that $H_{n} \to \infty$ in probability that the event 
$$ \{ \exists i \leq M: S_{i} \mbox{ is a free spot}\} \cap \{ \# \{i  \leq \lfloor H_{n}/2 \rfloor  : S_{i} \mbox{ is a free spot}\} \geq 3 M\}$$
has probability at least $1- \delta$, provided that $M$ is large enough. We can then argue as above and apply Lemma \ref{lem:rough} to deduce that the third item on $ \mathrm{Down}_{n}$ holds with probability  asymptotically larger than $1 - \varepsilon/3$. 
\end{proof}
From the last lemma, the convergence $ \mathcal{T}_n \to \mathcal{T}_\infty$ as labeled trees, and the fact that $ \mathcal{T}_\infty$ has a single end, it is easy to see that we indeed have $ \varphi(\mathcal{T}_n) \to \varphi(\mathcal{T}_\infty)$ a.s. 

\section{Comments and extensions}
We mention here a few possible developments that we hope to pursue in the future.
\paragraph{On the critical case.} As mentioned in the introduction, probably the most interesting question is to study the critical case $\Theta=0$. 
We tackle this problem in a forthcoming work in the case of plane trees and study the scaling limit of the renormalized flux on $ \mathcal{T}_n$. In  ``generic" situations, the flux is of order $n^{1/3}$ on $ \mathcal{T}_n$ and the components of parked vertices form a stable tree of parameter $3/2$. The components themselves are described by the growth-fragmentation trees considered in \cite{BCK18} in the context of random planar maps.  We also find a one-parameter family of possible scaling limits when the car arrivals have a ``heavy tail" $\mu([k,\infty)) \sim c k^{-\alpha}$ with $\alpha \in (2,3)$, which is again linked to the growth-fragmentation trees considered in \cite{BBCK18}.
\paragraph{Sharpness of the phase transition.} It is natural to expect that the phase transition for the parking is ``sharp" in the sense that many observables undergo a drastic change when going from the subcritical to the supercritical regime : this has been verified by Contat \cite{C21}, who shows that  $\mathbb{P}( \varphi( \mathcal{T}_{n}) =0)$ decays exponentially fast in the supercritical regime, whereas  in the subcritical regime, $ \varphi( \mathcal{T})$ has an exponential tail in the linear scale, see Theorem 2 of \cite{C21} for a precise statement.
\paragraph{Near-critical dynamics.} Also, in the first line of the table in Theorem \ref{thm:main}, one could approach the critical case by letting $\Theta=\Theta(n)$ approach 1 with $n$ while looking at $\mathcal T_n$ and try to delimit the regimes where $\varphi(\mathcal T_n)/ \E[\varphi(\mathcal T_n)]$ has a random limit (the critical window that extends the critical regime)
%$\Theta=0$) 
or a non random limit (the so called near-critical regimes). We hope to be able to study the dynamical scaling limits of the parking process in the critical window and compare it with the multiplicative coalescent which appears when studying the creation of the giant component in Erd\"os--R\'enyi random graphs. A preprint by Contat and the first author \cite{CC20} carries out this project in the case of Cayley trees.

%\newpage %% AUTHOR: please comment out this line.  It serves only
%%%   to demonstrate both types of header
%
% line in daj-template.pdf

%\section{Expansion estimates}
%
% More of the body of your paper goes here~\cite{bergelson-johnson-moreira}.
%
%%%% AUTHOR: optional appendix here
%\appendix %% you may comment this out if no Appendix
%\section*{Appendix}
%\section{Improving the constants}
%Material is placed here as needed.

%%% AUTHOR: optional acknowledgments here
\section*{Acknowledgments} %%  you may comment this out if no Ackno
%\textbf{Acknowledgments.} 
We thank Bastien Mallein and Christina Goldschmidt for an interesting discussion. Also, we are indebted to Alice Contat for spotting a typo in the proof of Proposition \ref{prop:meanflux}, and,  together with Linxiao Chen, for pointing the necessity of the assumption $m \leq 1$ in Theorem \ref{thm:main}.  We are grateful to the referees for comments.

%%% AUTHOR:
%%% Bibliography goes here. Note that the arXiv cannot process bibtex
%%% or biber bibliographies.  Example of acceptable bibliograpy format:
\bibliographystyle{amsplain}

%% AUTHOR: You can generate such a bibliography from a .bib file by 
%% running pdflatex/bibtex/pdflatex/pdflatex and then pasting the .bbl file
%% between \begin{thebibliography} and \end{bibliography}

%%% AUTHOR: Include a short description of each author following the
%%% structure below. Use the same short tags used previously.  
%%% Use \imageat{} and \imagedot{} instead of "@" and "." in
%%% email addresses-this replaces the symbols with graphics to avoid 
%%% e-mail address harvesting from the .pdf file
\begin{dajauthors}
\begin{authorinfo}[nicolas]
  Nicolas Curien\\
 Université Paris-Saclay, CNRS\\
  Laboratoire de mathématiques d'Orsay \\
  91405, Orsay, France.\\
  nicolas\imagedot{}curien\imageat{}gmail\imagedot{}com \\
  \url{https://www.imo.universite-paris-saclay.fr/~curien/}
\end{authorinfo}
\begin{authorinfo}[olivier]
  Olivier H\'enard\\
  Université Paris-Saclay, CNRS\\
  Laboratoire de mathématiques d'Orsay \\
  91405, Orsay, France.\\
  olivier\imagedot{}henard\imageat{}universite-paris-saclay\imagedot{}fr \\
  \url{https://www.imo.universite-paris-saclay.fr/~henard/}
\end{authorinfo}
\end{dajauthors}

\end{document}